\documentclass[11pt,a4paper,twoside]{article}
\usepackage[colorlinks=true,linkcolor=blue,citecolor=blue]{hyperref}
\usepackage{amsmath,graphicx,multicol,amsfonts,mathtools,amsthm,epstopdf,fancyhdr,lastpage,
xcolor,titlesec,cleveref,mathrsfs,bbold}
\usepackage[english]{babel}
\usepackage[bottom,multiple]{footmisc}
\usepackage[margin=3cm]{geometry}
\usepackage[font=scriptsize,labelfont=bf]{caption,subcaption}
\titleformat{\section}[block]{\bfseries\large}{\thesection. }{2pt}{}

\theoremstyle{definition}
\newtheorem{theorem}{Theorem}[section]
\newtheorem{definition}[theorem]{Definition}
\newtheorem{proposition}[theorem]{Proposition}

\newtheorem{lemma}[theorem]{Lemma}

\allowdisplaybreaks
\everymath{\displaystyle}

\begin{document}
\thispagestyle{empty}
\begin{center}
\textbf{{\Large General One-dimensional Clifford Fourier Transform\\
and Applications to Probability Theory}}
\end{center}

\vspace{1cm}

\begin{center}
\textbf{Said FAHLAOUI}\footnote{Corresponding Author: s.fahlaoui@umi.ac.ma}$^,$\addtocounter{footnote}{+1}\footnotemark $\;$ and $\;$ \addtocounter{footnote}{-2}\textbf{Hakim MONAIM}\footnote{monaim.hakim@edu.umi.ac.ma}$^,$\footnote{Faculty of Sciences, Moulay Ismail University of Meknès, Morocco.}
\end{center}

\leftskip=2cm

\vspace{3cm}

 \noindent\textbf{Abstract.}
 In this paper, we present the general one-dimensional Clifford Fourier Transform. We derive fundamental properties: Plancherel theorem, reconstruction and convolution formulas. Additionally, we provide applications to probability theory illustrating results with examples.
 
\leftskip=0cm 

\vspace*{4cm}

\section{Introduction}

\quad Clifford algebra, also known as geometric algebra, is a mathematical stricture that extends the principles of complex numbers and vectors to higher dimensions. The idea was first introduced by the famous mathematician of the late $19^{th}$ century William Kingdon Clifford \cite{clifford1878applications}. The concept of Clifford algebra is based on the notion of a geometric product between vectors, which provides a robust tool for analyzing geometric objects in higher dimensions. Its application is far-reaching, spanning fields such as physics, engineering, and computer science.

One of the significant strengths of Clifford algebra is its ability to unify and simplify many areas of mathematics and physics, including linear algebra, differential geometry, electromagnetism, and quantum mechanics. In \cite{ref2}, authors aim to connect Clifford algebras, manifolds and harmonic analysis, and to demonstrate the fundamental role of algebra, geometry, and differential equations in Euclidean Fourier analysis. They also combined the representation theory of Euclidean space with the representation theory of semisimple Lie groups.

Several works have explored the applications of Clifford algebra in signal processing. For instance, in \cite{brackx2009fourier} Bracx et al. introduces the new Clifford-Fourier transform, with a focus on the 2D case. Todd ell proposes the Fourier transform over the algebra of quaternions $(\mathbb{H}\simeq C\ell_{0,2})$ in \cite{ref8}, called quaternionic Fourier transform (QFT), which he explores to analyze systems described by partial differential equations. Hitzer has also made significant contributions to the development of this theory. In his works, he examines the different forms of the quaternionic Fourier transform (QFT) and explores its application to quaternion fields, providing corresponding Plancherel theorems \cite{ref12}. He also derived a new directional uncertainty principle for quaternion-valued functions using the quaternionic Fourier transform in \cite{ref13}, and extends it to establish similar principles in Clifford geometric algebras with quaternion subalgebras. In \cite{ref14}, he explains the orthogonal planes split (OPS) of quaternions based on the choice of one or two linearly independent pure unit quaternions and systematically generalizes the quaternionic Fourier transform applied to quaternion fields to conform with the OPS, establishing inverse transformations and commenting on their geometric meaning. He generalized, in his chapters \cite{ref17,ref26}, the aforementioned split (OPS) to a freely steerable orthogonal 2D-planes split of two orthonormal and collinear pure unit quaternions. This general form of OPS allows new geometric interpretations of the action of the QFT on the signals. In their works \cite{lounesto1986clifford,ref4}, P. Lounesto and Bracx et al. provide a historical review of the development and applications of quaternion and Clifford algebra wavelets.

In addition to the above, Bahri et al. introduced, in \cite{bahri2019one,one}, the one-dimensional quaternion Fourier transform and have established its properties which generalizes the Fourier transform and studied its application in probability theory. The works on Clifford algebra, therefore, have significant implications in several fields, and their continued exploration promises to unlock further insights and advancements.
\section{The Clifford Geometric Algebra} \label{sec2}

The Clifford geometric algebra $C\ell(\mathbb{R}^{p, q})=C\ell_{p, q}$  over the $\mathbb{R}-$linear space $\mathbb{R}^{p,q}$, is a non-commutative algebra generated by the $\mathbb{R}^{p,q}$-orthonormal vector basis $\mathscr{B} = \{e_1,\cdots,e_n\}$ (with $p+q=n$) obeying to the following associative non-commutative geometric multiplication rules (see \cite{ref2,ref17})
\begin{align}\label{1}
e_\ell e_k + e_k e_\ell = 2\delta_{\ell,k}\varepsilon_{_\ell}
\end{align}
where $\delta_{\ell,k}$ is the Kronecker symbol and
\begin{align}\label{1.1}
\varepsilon_{_\ell} = \left[  \mathbb{1}_{[\![1,p]\!]}(\ell) - \mathbb{1}_{[\![p+1,n]\!]}(\ell)\right].
\end{align}
The Clifford geometric algebra $C\ell_{p, q}$ can be split into the following direct sum \cite{lounesto1986clifford,ref17}
\begin{align}\label{2}
C\ell_{p, q} = \bigoplus_{\ell =0}^nC\ell_{p, q}^\ell
\end{align}
where $C\ell_{p, q}^\ell$ denotes the space spanned by the $\ell$-vectors family
\begin{align}\label{3}
 \mathscr{B}_\ell = \{ e_{\sigma_1}e_{\sigma_2}\cdots e_{\sigma_\ell},\qquad 1\leq \sigma_1< \sigma_2<\cdots <\sigma_\ell\leq n \}.
 \end{align}
Therefore, the set
\begin{align}\label{4}
\{ e_{_\Sigma} = e_{\sigma_1}e_{\sigma_2}\cdots e_{\sigma_\ell}, \qquad \Sigma \subseteq[\![1,n]\!], \quad 1\leq \sigma_1< \sigma_2<\cdots <\sigma_\ell\leq n\}\cup \{e_\emptyset = 1\}
\end{align}
forms a graded (blade) basis of $C\ell_{p, q}$. The grades $\ell$ range from 0 for scalars,
1 for vectors, 2 for bivectors, $\ell$ for $\ell$-vectors, up to n for pseudo-scalars. The the field $\mathbb{R}$ (resp. $\mathbb{R}$-linear space $\mathbb{R}^{p,q}$) is included in $C\ell_{p, q}$ as the subset of $0$-vectors (resp. $1$-vectors).\\
 The Clifford product (\ref{1}) generates a basis for $C\ell_{p,q}$ consisting of $2^n$ elements. A general element $\mathcal{C}$ of $ C\ell_{p, q}$ (called Clifford numbers, multivectors or hypercomplex numbers) is a real linear combination of basis blades $ (e_{_\Sigma})_{_\Sigma} $ and can be expanded as \cite{ref17}
\begin{align}\label{5}
 \mathcal{C} = \sum_{\Sigma\subseteq [\![1,n]\!]}\mathcal{C}_{_\Sigma} e_{_\Sigma} = \overbrace{\mathcal{C}_\emptyset}^{\text{scalar part}} + \overbrace{\sum_{\ell\in [\![1,n]\!]}\mathcal{C}_\ell e_\ell}^{\text{vector part}} + \overbrace{\sum_{1\leq \ell,k\leq n}\mathcal{C}_{\ell k} e_\ell e_k}^{\text{bivector part}} +\cdots+ \overbrace{\mathcal{C}_{12\cdots n} e_1e_2\cdots e_n}^{\text{pseudo-scalar part}} 
\end{align}
where $\mathcal{C}_{_\Sigma}$ are real-valued coefficients. $\mathcal{C}$ can also be written as
\begin{align}\label{6}
 \mathcal{C} = \sum_{\ell=0}^n \langle \mathcal{C} \rangle_{_\ell} = \langle \mathcal{C} \rangle_{_0} + \langle \mathcal{C} \rangle_{_1} +\cdots+ \langle \mathcal{C} \rangle_{_n}
\end{align}
with $ \langle \mathcal{C} \rangle_{_\ell} = \sum_{\#\Sigma=\ell} \mathcal{C}_{_\Sigma}e_{_\Sigma} $ denotes the $\ell$-vectors part of $\mathcal{C}$. As examples, $ \langle \mathcal{C} \rangle_{_0} $ denotes the scalar part, $ \langle \mathcal{C} \rangle_{_1} $ the vector part, $ \langle \mathcal{C} \rangle_{_2} $ the bi-vector part and $ \langle \mathcal{C} \rangle_{_n} $ the pseudo-scalar part.\\
The principal reverse of a multi-vector $\mathcal{C}\in C\ell(p, q)$ is defined as \cite{lounesto1986clifford,ref17}
\begin{align}\label{7}
 \widetilde{\mathcal{C}} = \sum_{\ell=0}^n (-1)^{\frac{\ell(\ell-1)}{2}}\overline{\langle \mathcal{C} \rangle_{_\ell}} 
\end{align}
where $\quad \overline{ \mathcal{C}} \quad $ means to change in the basis decomposition of $\quad \mathcal{C} \quad $ the sign of every vector of negative square $\quad \overline{e_{_\Sigma}} = \varepsilon_{_{\sigma_1}}e_{\sigma_1}\varepsilon_{_{\sigma_2}}e_{\sigma_2}\cdots\varepsilon_{_{\sigma_\ell}}e_{\sigma_\ell} \quad $ where $\; 1\leq \sigma_1< \sigma_2<\cdots<\sigma_\ell\leq n$ and $\; \varepsilon_{_{\sigma_k}} \; $ is given by (\ref{1.1}).\\
The principal reverse is linear, involution and anti-automorphic, that is for all $\mathcal{C}, \mathcal{D} \in C\ell_{p, q}$ 
\begin{align}\label{8}
 \widetilde{\mathcal{C}+\mathcal{D}} = \widetilde{\mathcal{C}} + \widetilde{\mathcal{D}}, \qquad \widetilde{\widetilde{\mathcal{C}}} = \mathcal{C}, \qquad \widetilde{\mathcal{C}\mathcal{D}} = \widetilde{\mathcal{D}} \widetilde{\mathcal{C}}.
\end{align}
The scalar product of $\mathcal{C}, \mathcal{D} \in C\ell_{p, q}$ can be defined by \cite{ref17}
\begin{align}\label{9}
 \mathcal{C}*\widetilde{\mathcal{D}} = \langle\mathcal{C}\widetilde{\mathcal{D}}\rangle_{_0} = \sum_{\Sigma\subseteq [\![1,n]\!]} \mathcal{C}_{_\Sigma}\mathcal{D}_{_\Sigma}. 
\end{align}
In particular, if $ \mathcal{C}=\mathcal{D} $, then the modulus of a multi-vector $ \mathcal{C}\in C\ell_{p, q} $ is given by \cite{lounesto1986clifford,ref17}
\begin{align}\label{10}
 |\mathcal{C}| = \sqrt{\langle \mathcal{C}\widetilde{\mathcal{C}} \rangle_0} = \sqrt{\sum_{\Sigma\subseteq [\![1,n]\!]} \mathcal{C}_{_\Sigma}^2} .
\end{align}
For $\mathcal{C}, \mathcal{D} \in C\ell_{p, q}\quad (p+q=n\geq 3)$, the following property holds \cite{donoho}
\begin{align}\label{11}
 |\mathcal{C}\mathcal{D}| \leq 2^n|\mathcal{C}||\mathcal{D}| .
\end{align}
\noindent Inner product on the square-integrable Clifford geometric algebra valued-function space $ f,g \in L^2(\mathbb{R},C\ell_{p, q}) $ is defined as follow
\begin{align}\label{13}
(f,g)_{L^2(\mathbb{R},C\ell_{p, q})} = \int_\mathbb{R}f(x)\overline{g(x)}dx .
\end{align}
For $ f=g $, we get the $ L^2(\mathbb{R},C\ell_{p, q})$-norm as
\begin{align}\label{14}
\|f\|_{L^2(\mathbb{R},C\ell_{p, q})}^2 = \int_\mathbb{R}|f(x)|^2dx .
\end{align}
\section{The general one-dimensional Clifford Fourier transform}
Let’s denote $\mathbb{C}_\mu$ the 2D sub-plane of $ C\ell_{p, q} $ spanned by $\{1,\mu\}$, where $\mu \in C\ell_{p, q}\; \text{and} \; \mu^2 = -1$
\begin{align}\label{11.1}
\mathbb{C}_\mu = span\{1,\mu\} = \{a+b\mu \; , \quad a,b \in \mathbb{R}  \}
\end{align}
$\mathbb{C}_\mu$ is an algebraically closed commutative field isomorphic to the complex plane $\mathbb{C}$. Each hypercomplexe number $q\in \mathbb{C}_\mu$ can be written, in the polar form, as \cite{donoho}
\begin{align}\label{12}
 q = |q| e^{\theta\mu} = |q|\left(\cos\theta + \mu\sin\theta\right) .
\end{align}
\begin{definition}
Let $\mu\in C\ell_{p, q}\quad \text{with} \quad \mu^2 = -1$. The general one-dimensional Clifford Fourier transform (1DCFT) of $ f\in L^1(\mathbb{R},C\ell_{p, q}) $, with respect to $\mu$ is given by
\begin{align}\label{15}
 \mathcal{F}^\mu(f)(\xi) = \int_\mathbb{R}f(x)e^{\mu x\xi} dx
\end{align}
where $x,\xi \in \mathbb{R} $.
\end{definition}
\noindent If we use the $C\ell_{p, q}$-basis expansion; $ f = \sum_{\Sigma\subseteq [\![1,n]\!]} e_{_\Sigma}f_{_\Sigma}$, 1DCFT of $f$ becomes
\begin{align}\label{15.}
 \mathcal{F}^\mu(f)(\xi) = \sum_{\Sigma\subseteq [\![1,n]\!]} e_{_\Sigma}\mathcal{F}^\mu(f_{_\Sigma})(\xi).
\end{align}
We present in the following fundamental properties of the 1DCFT, for their proofs and more comprehensive analysis, refer to \cite{bahri2019one,hitzer2012clifford,hitzer2008clifford},
\begin{itemize}
\item[-] For all $\mathcal{C},\mathcal{D} \in C\ell_{p, q}$ and $f,g \in L^1(\mathbb{R},C\ell_{p, q})$ we get
\begin{align}\label{16}
 \mathcal{F}^\mu (\mathcal{C}f + \mathcal{D}g)(\xi) = \mathcal{C}\mathcal{F}^\mu (f)(\xi) + \mathcal{D}\mathcal{F}^\mu (g)(\xi).
\end{align}

\item[-] For all $f \in L^1(\mathbb{R},C\ell_{p, q})$, and $ h \in \mathbb{R} $ we have
\begin{align}\label{16}
 \mathcal{F}^\mu(\tau_{_h}f)(\xi) = \mathcal{F}^\mu (f)(\xi) e^{-\mu h\xi}
\end{align}
where the translation operator is given by $ \tau_{_h}f(x) = f(x+h) $.\\

\item[-] For all $f,\mathcal{F}^\mu(f) \in L^1(\mathbb{R},C\ell_{p, q})$, f is recovered from its Fourier transform as
\begin{align}\label{17}
 f(x) = \frac{1}{2\pi}\int_{\mathbb{R}}\mathcal{F}^\mu(f)(\xi)e^{-\mu \xi x} d\xi .
\end{align}

\item[-] For all $f,\mathcal{F}^\mu(f) \in L^2(\mathbb{R},C\ell_{p, q})$, Parseval’s identity holds
\begin{align}\label{18}
 2\pi \|f\|_2 = \|\mathcal{F}^\mu(f)\|_2 .
\end{align}

\item[-] For $f \in \mathscr{C}^m(\mathbb{R},C\ell_{p, q})$
\begin{align}\label{22}
 \mathcal{F}^\mu\left(\frac{d^m}{dx^m}f\right)(\xi) = \mathcal{F}^\mu\left(f\right)(\xi)(-\mu\xi)^m.
\end{align}
\end{itemize}
\begin{definition}
The convolution of two Clifford algebra valued functions $f,g \in L^1(\mathbb{R},C\ell_{p, q})$ is defined by
\begin{align}\label{19}
 f*g(y) = \int_\mathbb{R}f(x)g(y-x)dx.
\end{align}
\end{definition}
\begin{proposition} Let $f,g \in L^1(\mathbb{R},C\ell_{p, q})$. We have
\begin{align}\label{20}
\mathcal{F}^\mu(f*g)(\xi) = \sum_{\Sigma\subseteq [\![1,n]\!]}\mathcal{F}^\mu(fe_{_\Sigma})(\xi)\mathcal{F}^\mu(g_{_\Sigma})(\xi).
\end{align}
\end{proposition}
\begin{proof}
If we use the expansion (\ref{5}) of $g$ 
\begin{align}\label{21}
 g = \sum_{\Sigma\subseteq [\![1,n]\!]} e_{_\Sigma}g_{_\Sigma},
\end{align}
we get
\begin{align*}
\mathcal{F}^\mu(f*g)(\xi) &= \int_{\mathbb{R}} \int_\mathbb{R}f(y)g(x-y)dy e^{\mu x\xi} dx \\
&= \int_{\mathbb{R}} \int_\mathbb{R}f(y)\sum_{\Sigma\subseteq [\![1,n]\!]} e_{_\Sigma}g_{_\Sigma}(x-y)dy e^{\mu x\xi} dx \\
&= \sum_{\Sigma\subseteq [\![1,n]\!]}\int_{\mathbb{R}} f(y)e_{_\Sigma}e^{\mu y\xi}\int_\mathbb{R}g_{_\Sigma}(x-y) e^{\mu (x-y)\xi}dx dy \\
&= \sum_{\Sigma\subseteq [\![1,n]\!]}\mathcal{F}^\mu(fe_{_\Sigma})(\xi)\mathcal{F}^\mu(g_{_\Sigma})(\xi).
\end{align*}
\begin{flushright}
$\Box$
\end{flushright}
\end{proof}
\begin{lemma} (Riemann-Lebesgue)
The general one-dimensional Clifford Fourier transform (1DCFT) $\mathcal{F}^\mu$ maps $\mathbf{L}^1(\mathbb{R},C\ell_{p, q})$ into $\mathscr{C}_0(\mathbb{R},C\ell_{p, q})$ and it is one-to-one. Where $\mathscr{C}_0(\mathbb{R},C\ell_{p, q})$ is the set of continuous functions vanishing at infinity \cite{ref18}.
\end{lemma}
\begin{theorem} (Szokefalvi-Nagy’s inequality). Let $ f\in L^2(\mathbb{R},C\ell_{p, q}) $, then
\begin{align}\label{23}
 |f(x)|^2 \leq \|f\|_2\left\|\frac{df}{dx}\right\|_2 .
\end{align}
\end{theorem}
\begin{proof}\\
$ \textit{\textbf{i-}} $ $ \frac{df}{dx} \notin L^2(\mathbb{R},C\ell_{p, q}) $, the Szokefalvi-Nagy’s inequality abviously holds.\\
$ \textit{\textbf{ii-}} $ $ \frac{df}{dx} \in L^2(\mathbb{R},C\ell_{p, q}) $. For $ \rho > 0 $, one gets
\begin{align*}
|f(x)|^2 &= \frac{1}{4\pi^2}\left| \int_{\mathbb{R}}\mathcal{F}^\mu(f)(\xi)e^{\mu \xi x} d\xi \right|^2 \\
&\leq \frac{1}{4\pi^2} \int_{\mathbb{R}}(\rho + \xi^2) \left|\mathcal{F}^\mu(f)(\xi)\right|^2\frac{e^{2\mu \xi x}}{\rho + \xi^2}  d\xi .
\end{align*}
Cauchy-Schwartz inequality, (\ref{18}) and (\ref{22}) give
\begin{align*}
|f(x)|^2 &\leq \frac{1}{4\pi^2} \int_{\mathbb{R}}(\rho + \xi^2) \left|\mathcal{F}^\mu(f)(\xi)\right|^2d\xi \int_{\mathbb{R}}\frac{1}{\rho + \xi^2} d\xi  \\
&\leq \frac{1}{4\pi\sqrt{\rho}} \int_{\mathbb{R}}(\rho + \xi^2) \left|\mathcal{F}^\mu(f)(\xi)\right|^2d\xi \\
&\leq \frac{1}{4\pi\sqrt{\rho}} \left[\int_{\mathbb{R}}\rho \left|\mathcal{F}^\mu(f)(\xi)\right|^2d\xi + \int_{\mathbb{R}}\xi^2 \left|\mathcal{F}^\mu(f)(\xi)\right|^2d\xi \right] \\
&\leq \frac{1}{4\pi\sqrt{\rho}} \left[\int_{\mathbb{R}}2\pi\rho \left|f(x)\right|^2dx + \int_{\mathbb{R}}2\pi \left|\frac{d}{dx}f(x)\right|^2d\xi \right] \\
&\leq \frac{\sqrt{\rho}}{2} \|f\|_2^2 + \frac{1}{2\sqrt{\rho}}\left\|\frac{df}{dx}\right\|_2^2 .
\end{align*}
Setting $ \rho = \|f\|_2^{-2}\left\|\frac{df}{dx}\right\|_2^2 $, we obtain
\begin{align}\label{24}
 |f(x)|^2 \leq \|f\|_2\left\|\frac{df}{dx}\right\|_2.
\end{align}

\begin{flushright}
$\Box$
\end{flushright}
\end{proof}

\section{One-Dimensional Clifford Fourier Transform in Probability Theory}
\begin{definition}
A Clifford algebra-valued function $ f_X(x) = \sum_{\Sigma\subseteq [\![1,n]\!]}e_{_\Sigma}(f_X)_{_\Sigma}(x)  $ is called the Clifford algebra probability density function of a real random variable $X$ if $\;\forall \; \Sigma\subseteq [\![1,n]\!]$
\begin{align}\label{25}
\int_\mathbb{R}(f_X)_{_\Sigma}(x)dx = 1 \qquad \text{and} \qquad \{(f_X)_{_\Sigma}<0\} = \emptyset.
\end{align}
Here, $ (f_X)_{_\Sigma} $ is a real probability density function. The Clifford algebra cumulative distribution function is expressed as
\begin{align}\label{26}
 f_X(x) = \frac{d}{dx}F_X(x), 
\end{align}
where the probability P is related to $F_X$ given by
\begin{align}\label{27}
F_X(x) = P(X \leq x).
\end{align}
\end{definition}

\begin{definition} Let $X$ be a real random variable with the Clifford Algebra probability density function $f_X$. The $\ell^{th}$ moment of $X$ is defined as
\begin{align}\label{34}
 m_\ell = E[X^\ell] = \int_\mathbb{R}x^\ell f_X(x)dx .
\end{align}
\end{definition}
\noindent If we set
\[  f_X(x) = \sum_{\Sigma\subseteq [\![1,n]\!]}e_{_\Sigma}(f_X)_{_\Sigma}(x) \qquad \qquad \text{and} \qquad \qquad \int_\mathbb{R}x^\ell (f_X)_{_\Sigma}(x) dx := E[X_{_\Sigma}^\ell]= (m_\ell)_{_\Sigma} ,\] 
we get
\begin{align*}
m_\ell &= \int_\mathbb{R}x^\ell \sum_{\Sigma\subseteq [\![1,n]\!]}e_{_\Sigma}(f_X)_{_\Sigma}(x)  dx \\
&= \sum_{\Sigma\subseteq [\![1,n]\!]} e_{_\Sigma}\int_\mathbb{R}x^\ell (f_X)_{_\Sigma}(x) dx  \\
&= \sum_{\Sigma\subseteq [\![1,n]\!]} e_{_\Sigma}E[X_{_\Sigma}^\ell]\\
&= \sum_{\Sigma\subseteq [\![1,n]\!]} e_{_\Sigma}(m_\ell)_{_\Sigma} .
\end{align*}
\noindent It is easily seen that
\begin{align}\label{28}
 |m_\ell|^2 = E[X]\overline{E[X]} = \sum_{\Sigma\subseteq [\![1,n]\!]} (m_\ell)_{_\Sigma}^2.
\end{align}
\begin{definition} Let $X$ be a real random variable with the Clifford algebra probability density function $f_X$. The characteristic function of $X,\quad \phi_X : \mathbb{R} \longrightarrow C\ell_{p, q})$, is defined by the formula (compare with (\ref{11.1}))
\begin{align}\label{29}
 \phi_X(t) = E[e^{\mu tX}] = \int_\mathbb{R}f_X(x)e^{\mu tx} dx = \mathcal{F}^\mu(f_X)(t).
\end{align}
\end{definition}
\noindent Setting $ f_X(x) = \sum_{\Sigma\subseteq [\![1,n]\!]}e_{_\Sigma}(f_X)_{_\Sigma}(x)$, the characteristic function of $X$ can be expressed as
\begin{align*}
\phi_X(t) &= \mathcal{F}^\mu(f_X)(t) \\
&= \int_\mathbb{R}f_X(x)e^{\mu tx} dx \\
&= \int_\mathbb{R}\sum_{\Sigma\subseteq [\![1,n]\!]}e_{_\Sigma}(f_X)_{_\Sigma}(x) e^{\mu tx} dx\\
&= \sum_{\Sigma\subseteq [\![1,n]\!]}e_{_\Sigma}\int_\mathbb{R}(f_X)_{_\Sigma}(x) e^{\mu tx} dx\\
&= \sum_{\Sigma\subseteq [\![1,n]\!]}e_{_\Sigma}\int_\mathbb{R}(f_X)_{_\Sigma}(x) e^{\mu tx} dx\\
&= \sum_{\Sigma\subseteq [\![1,n]\!]}e_{_\Sigma}\mathcal{F}^\mu((f_X)_{_\Sigma})(t)\\
&= \sum_{\Sigma\subseteq [\![1,n]\!]}e_{_\Sigma}(\phi_X)_{_\Sigma}(t).
\end{align*}
By inversion formula, we get
\begin{align}\label{30}
 f_X(x) = \int_\mathbb{R}\phi_X(x)e^{-\mu tx} dx = \mathcal{F}^{-\mu}(\phi_X)(t).
\end{align}
From (\ref{29}), (\ref{26}) and Definition (\ref{22}), one gets 
\begin{align*}
\phi_X(t) &= \mathcal{F}^\mu(f_X)(t) \\
&= \mathcal{F}^\mu\left(\frac{d}{dx}F_X\right)(t) \\
&= -\mathcal{F}^\mu\left(F_X\right)(t)\mu t .
\end{align*}
Thus, for $t\neq 0$
\begin{align}\label{31}
\mathcal{F}^\mu\left(F_X\right)(t) = \frac{1}{t}\phi_X(t)\mu .
\end{align}
\begin{theorem} Let $\phi_X$ and $\psi_X$ be two Clifford Algebra characteristic functions of the random variable $X$, given by
\begin{align}\label{32}
 \phi_X(t) = \int_\mathbb{R}f_X(x)e^{\mu tx} dx \qquad \text{and} \qquad \psi_X(x) = \int_\mathbb{R}g_X(t)e^{\mu tx} dt ,
\end{align}
then
\begin{align}\label{33}
 \int_\mathbb{R}g_X(t)\phi_X(t)e^{-\mu ty} dt = \sum_{\Sigma\subseteq [\![1,n]\!]} e_{_\Sigma}f_X*(\psi_X)^\lozenge_{_\Sigma}(y),
\end{align}
with $\quad (\psi_X)^\lozenge_{_\Sigma}(x) = (\psi_X)_{_\Sigma}(-x)$
\end{theorem}
\begin{proof} Let's expand $g_X$ on $C\ell_{p, q}$-basis; $\quad g_X(x) = \sum_{\Sigma\subseteq [\![1,n]\!]}e_{_\Sigma}(g_X)_{_\Sigma}(x)$. Forward calculations yield
\begin{align*}
\int_\mathbb{R}g_X(t)\phi_X(t)e^{\mu tx} dt &= \int_\mathbb{R}g_X(t)\left( \int_\mathbb{R}f_X(x)e^{\mu tx} dx \right) e^{-\mu ty} dt \\
&= \int_\mathbb{R}g_X(t)\left( \int_\mathbb{R}f_X(x)e^{\mu t(x-y)} dx \right) dt \\
&= \int_\mathbb{R}\sum_{\Sigma\subseteq [\![1,n]\!]}e_{_\Sigma}(g_X)_{_\Sigma}(t) \left( \int_\mathbb{R}f_X(x)e^{\mu t(x-y)} dx \right) dt \\
&= \sum_{\Sigma\subseteq [\![1,n]\!]} e_{_\Sigma}\int_\mathbb{R} f_X(x)\int_\mathbb{R}(g_X)_{_\Sigma}(t)e^{\mu t(x-y)} dtdx  \\
&= \sum_{\Sigma\subseteq [\![1,n]\!]} e_{_\Sigma}\int_\mathbb{R} f_X(x)(\psi_X)_{_\Sigma}(x-y)dx. \\
&= \sum_{\Sigma\subseteq [\![1,n]\!]} e_{_\Sigma}f_X*(\psi_X)^\lozenge_{_\Sigma}(y). 
\end{align*}
\begin{flushright}
$\Box$
\end{flushright}
\end{proof}

\begin{theorem} If $X$ is a real random variable, then there exists $\ell^{th}$ derivatives for the Clifford Algebra characteristic function $\phi_X$ which is given by the formula
\begin{align}\label{35}
 \frac{d^\ell}{dt^\ell}\phi_X(t) = \int_\mathbb{R}x^\ell f_X(x)e^{\mu tx} dx\mu^\ell .
\end{align}
Moreover
\begin{align}\label{36}
 m_\ell = E[X^\ell] = \frac{d^\ell}{dt^\ell}\phi_X(0)(-\mu)^\ell .
\end{align}
\end{theorem}
\begin{proof} For $\ell = 1$, direct computations reveal that
\begin{align*}
\frac{d}{dt}\phi_X(t) &= \frac{d}{dt}\left(\int_\mathbb{R} f_X(x)e^{\mu tx} dx\right) \\
&= \int_\mathbb{R} f_X(x)\frac{d}{dt}\left(e^{\mu tx} \right)dx \\
&= \int_\mathbb{R} f_X(x)e^{\mu tx} xdx\mu .
\end{align*}
Suppose that
\begin{align}\label{37}
 \frac{d^{\ell-1}}{dt^{\ell-1}}\phi_X(t) = \int_\mathbb{R}x^{\ell-1} f_X(x)e^{\mu tx} dx\mu^{\ell-1} .
\end{align}
We have
\begin{align*}
\frac{d^\ell}{dt^\ell}\phi_X(t) &= \frac{d}{dt}\left(\frac{d^{\ell-1}}{dx^{\ell-1}}\phi_X(t)\right) \\
&= \frac{d}{dt}\left(\int_\mathbb{R}x^{\ell-1} f_X(x)e^{\mu tx} dx\mu^{\ell-1}\right) \\
&= \int_\mathbb{R}x^{\ell-1} f_X(x)\frac{d}{dt}\left(e^{\mu tx} \right)dx\mu^{\ell-1} \\
&= \int_\mathbb{R}x^\ell f_X(x)e^{\mu tx} dx\mu^\ell .
\end{align*}
Hence
\begin{align}\label{38}
 \frac{d^\ell}{dt^\ell}\phi_X(t)(-\mu)^\ell = \int_\mathbb{R}x^\ell f_X(x)e^{\mu tx} dx.
\end{align}
Then
\begin{align}\label{39}
 m_\ell = E[X^\ell] = \frac{d^\ell}{dt^\ell}\phi_X(0)(-\mu)^\ell .
\end{align}
\begin{flushright}
$\Box$
\end{flushright}
\end{proof}
\begin{definition} The variance in the Clifford Algebra setting of a real random variable $X$ is defined by
\begin{align}\label{40}
 \sigma^2 = m_2 - m_1^2 = E[X^2] - (E[X])^2 .
 \end{align}
\end{definition}
\noindent By (\ref{39}), the variance $\sigma$ of $X$ in terms of the Clifford Algebra characteristic function can be expressed as
\begin{align}\label{41}
\sigma^2 = \frac{d^2}{dt^2}\phi_X(0)(-\mu)^2 - \left[\frac{d}{dt}\phi_X(0)(-\mu)\right]^2 = \left[\frac{d}{dt}\phi_X(0)\right]^2-\frac{d^2}{dt^2}\phi_X(0) .
\end{align}
\begin{theorem} Let $X$ be a real random variable with the Clifford algebra probability density function $f_X = \sum_{\Sigma\subseteq [\![1,n]\!]}e_{_\Sigma}f_{_\Sigma}$. Then $ \forall \quad \Sigma\subseteq [\![1,n]\!] $
\begin{align}\label{41.1}
1 \leq \left\|\frac{\partial}{\partial x}\ln(f_{_\Sigma})\right\|_2^2 \left\|\xi (\phi_X)_{_\Sigma} \right\|_2^2 (m_2)_{_\Sigma}.
\end{align}
\end{theorem}
\begin{proof}
Let
\begin{align}\label{41.2}
f_X = \sum_{\Sigma\subseteq [\![1,n]\!]}e_{_\Sigma}f_{_\Sigma} = \sum_{\Sigma\subseteq [\![1,n]\!]}e_{_\Sigma}g^2_{_\Sigma}.
\end{align}
The Heisenberg uncertainty principle \cite{chandrasekharan2012classical} gives
\begin{align}\label{41.4}
\frac{\|g\|_2^4}{4} \leq \| \xi\mathcal{F}^\mu(g_{_\Sigma}) \|_2^2 \|xg_{_\Sigma}\|_2^2 .
\end{align}
We have
\begin{align}\label{41.3}
\|f\|_1 = \|g\|_2^2 = 1 ,
\end{align}
and
\begin{align}\label{41.5}
\|xg_{_\Sigma}\|_2^2 = \int_{\mathbb{R}}x^2g^2_{_\Sigma}(x)dx = \int_{\mathbb{R}}x^2f_{_\Sigma}(x)dx = (m_2)_{_\Sigma}.
\end{align}
By Parseval identity (\ref{18})
\begin{align*}
\left\|\xi\mathcal{F}^\mu\left(g_{_\Sigma}\right) \right\|_2^2 &= \left\|\mathcal{F}^\mu\left(\frac{\partial}{\partial x}g_{_\Sigma}\right) \right\|_2^2 \\
 &= 2\pi \left\|\frac{\partial}{\partial x}g_{_\Sigma} \right\|_2^2 \\
&= \frac{\pi}{2} \left\|\frac{1}{\sqrt{f_{_\Sigma}}}\frac{\partial}{\partial x}f_{_\Sigma} \right\|_2^2 \\
&= \frac{\pi}{2} \left\|\frac{\partial}{\partial x}\ln(f_{_\Sigma})\frac{\partial}{\partial x}f_{_\Sigma} \right\|_1 .
\end{align*}
Since
\begin{align*}
\left\|\frac{\partial}{\partial x}\ln(f_{_\Sigma})\frac{\partial}{\partial x}f_{_\Sigma} \right\|_1 &\leq \left\|\frac{\partial}{\partial x}\ln(f_{_\Sigma})\right\|_2^2 \left\|\frac{\partial}{\partial x}f_{_\Sigma} \right\|_2^2 \\
&\leq \frac{1}{2\pi} \left\|\frac{\partial}{\partial x}\ln(f_{_\Sigma})\right\|_2^2 \left\|\mathcal{F}^\mu\left(\frac{\partial}{\partial x}f_{_\Sigma}\right) \right\|_2^2 .
\end{align*}  
Then (\ref{41.4}) becomes
\begin{align*}
1 &\leq \left\|\frac{\partial}{\partial x}\ln(f_{_\Sigma})\right\|_2^2 \left\|\mathcal{F}^\mu\left(\frac{\partial}{\partial x}f_{_\Sigma}\right) \right\|_2^2 (m_2)_{_\Sigma} \\
&\leq \left\|\frac{\partial}{\partial x}\ln(f_{_\Sigma})\right\|_2^2 \left\|\xi\mathcal{F}^\mu\left(f_{_\Sigma}\right) \right\|_2^2 (m_2)_{_\Sigma} \\
&\leq \left\|\frac{\partial}{\partial x}\ln(f_{_\Sigma})\right\|_2^2 \left\|\xi (\phi_X)_{_\Sigma} \right\|_2^2 (m_2)_{_\Sigma} .
\end{align*}
\begin{flushright}
$\Box$
\end{flushright}
\end{proof} 

\section{Examples}
\textit{\textbf{i-}} Consider a real random variable $X$ that can occur according to a Clifford algebra uniform law
\begin{align}\label{42}
f_X(x) = \sum_{\Sigma\subseteq [\![1,n]\!]}e_{_\Sigma}\mathbb{1}_{\left[\alpha_{_\Sigma},\beta_{_\Sigma}\right]} .
\end{align}
We have
\begin{center}
 $ \mathcal{F}^\mu\left( \mathbb{1}_{\left[\alpha_{_\Sigma},\beta_{_\Sigma}\right]} \right)(t) = \left\{
\begin{aligned}
\quad & \beta_{_\Sigma} - \alpha_{_\Sigma} \hspace{4.27cm} \text{if} \quad t = 0\\
\quad & \frac{2}{t}\sin\left((\beta_{_\Sigma} - \alpha_{_\Sigma})\frac{t}{2}\right) e^{\mu(\beta_{_\Sigma} + \alpha_{_\Sigma})\frac{t}{2}} \quad \text{if} \quad t \neq 0 .
\end{aligned}
\right.$
 \end{center}
It follows from (\ref{32}) that
\begin{center}
 $ \phi_X(t) = \mathcal{F}^\mu\left(f_X\right)(t) = \left\{
\begin{aligned}
\quad & \sum_{\Sigma\subseteq [\![1,n]\!]}e_{_\Sigma}\beta_{_\Sigma} - \alpha_{_\Sigma} \hspace{4.27cm} \text{if} \quad t = 0\\
\quad & \sum_{\Sigma\subseteq [\![1,n]\!]}e_{_\Sigma}\frac{2}{t}\sin\left((\beta_{_\Sigma} - \alpha_{_\Sigma})\frac{t}{2}\right) e^{\mu(\beta_{_\Sigma} + \alpha_{_\Sigma})\frac{t}{2}} \quad \text{if} \quad t \neq 0 .
\end{aligned}
\right.$
 \end{center}
The first and second derivatives of each real-valued coefficient of $\phi_X$ are given by
 \begin{align}\label{43}
\frac{d}{dt}(\phi_X)_{_\Sigma}(0) = \mu\frac{\beta_{_\Sigma}^2-\alpha_{_\Sigma}^2}{2}, \qquad \text{and} \qquad \frac{d^2}{dt^2}(\phi_X)_{_\Sigma}(0) = \frac{\alpha_{_\Sigma}^3-\beta_{_\Sigma}^3}{3}.
\end{align}
 Then
\begin{align}\label{45}
m_1 = \frac{d}{dt}\phi_X(0)(-\mu) = \sum_{\Sigma\subseteq [\![1,n]\!]}e_{_\Sigma}\frac{\beta_{_\Sigma}^2-\alpha_{_\Sigma}^2}{2},
\end{align}
 and
\begin{align}\label{46}
m_2 = \frac{d^2}{dt^2}\phi_X(0)(-\mu)^2 = \sum_{\Sigma\subseteq [\![1,n]\!]}e_{_\Sigma}\frac{\beta_{_\Sigma}^3-\alpha_{_\Sigma}^3}{3}.
\end{align}
By (\ref{41}), we get
\begin{align}\label{47}
\sigma^2 = m_2 - m_1^2 = \sum_{\Sigma\subseteq [\![1,n]\!]}e_{_\Sigma}\frac{\beta_{_\Sigma}^3-\alpha_{_\Sigma}^3}{3} - \left( \sum_{\Sigma\subseteq [\![1,n]\!]}e_{_\Sigma}\frac{\beta_{_\Sigma}^2-\alpha_{_\Sigma}^2}{2} \right)^2 .
\end{align}
\textit{\textbf{ii-}} Let $Y$ be a real random variable that has the probability density function
\begin{align}\label{48} 
g_Y(x) = \sum_{\Sigma\subseteq [\![1,n]\!]}e_{_\Sigma}\sqrt{\frac{\lambda_{_\Sigma}}{\pi}}e^{-\lambda_{_\Sigma}x^2},
\end{align}
where $ \left(\lambda_{_\Sigma}\right)_{\Sigma\subseteq [\![1,n]\!]} $ is a finite sequence of strictly positive real numbers.\\
It follows from (\ref{32}) that
\begin{align}\label{49}
\phi_Y(t) &= \mathcal{F}^\mu\left(g_Y\right)(t) = \sum_{\Sigma\subseteq [\![1,n]\!]}e_{_\Sigma}\sqrt{\frac{\lambda_{_\Sigma}}{\pi}}\int_\mathbb{R}e^{-\lambda_{_\Sigma}x^2}e^{\mu tx} dx = \sum_{\Sigma\subseteq [\![1,n]\!]}e_{_\Sigma}e^{-\frac{t^2}{4\lambda_{_\Sigma}}}.
\end{align}
The first and second derivatives of each real-valued coefficient of $\phi_Y$ are given as
\begin{align}\label{50}
\frac{d}{dt}(\phi_Y)_{_\Sigma}(t) = -\frac{t}{2\lambda_{_\Sigma}}e^{-\frac{t^2}{4\lambda_{_\Sigma}}},
\end{align}
and
\begin{align}\label{51}
\frac{d^2}{dt^2}(\phi_Y)_{_\Sigma}(t) = \left( \frac{t^2}{4\lambda_{_\Sigma}^2} - \frac{1}{2\lambda_{_\Sigma}} \right) e^{-\frac{t^2}{2\lambda_{_\Sigma}}}.
\end{align}
Then
\begin{align}\label{52}
m_1 = \frac{d}{dt}\phi_Y(0)(-\mu) = 0,
\end{align}
and
\begin{align}\label{53}
m_2 = \frac{d^2}{dt^2}\phi_Y(0)(-\mu)^2 = \sum_{\Sigma\subseteq [\![1,n]\!]}e_{_\Sigma} \frac{1}{2\lambda_{_\Sigma}}.
\end{align}
From (\ref{50},\ref{51}), we conclude
\begin{align}\label{54}
\sigma^2 = m_2 - m_1^2 = \sum_{\Sigma\subseteq [\![1,n]\!]}e_{_\Sigma} \frac{1}{2\lambda_{_\Sigma}}.
\end{align}
\textit{\textbf{iii-}} Let $Z$ be a real random variable that has the probability density function
\begin{align}\label{55} 
h_Z(x) = \sum_{\Sigma\subseteq [\![1,n]\!]}e_{_\Sigma}\lambda_{_\Sigma}e^{-\lambda_{_\Sigma}x}\mathbb{1}_{[0,+\infty[}(x) ,
\end{align}
where $ \left(\lambda_{_\Sigma}\right)_{\Sigma\subseteq [\![1,n]\!]} $ is a finite sequence of strictly positive real numbers.\\
It follows from (\ref{32}) that
\begin{align}\label{56}
\phi_Z(t) &= \mathcal{F}^\mu\left(h_Z\right)(t) = \sum_{\Sigma\subseteq [\![1,n]\!]}e_{_\Sigma}\int_\mathbb{R}\lambda_{_\Sigma}e^{-\lambda_{_\Sigma}x}\mathbb{1}_{[0,+\infty[}(x)e^{\mu tx} dx = \sum_{\Sigma\subseteq [\![1,n]\!]}e_{_\Sigma}\frac{\lambda_{_\Sigma}}{\lambda_{_\Sigma}-\mu t}.
\end{align}
The first and second derivatives of each real-valued coefficient of $\phi_Y$ are given as
\begin{align}\label{57}
\frac{d}{dt}(\phi_Z)_{_\Sigma}(t) = \frac{\lambda_{_\Sigma}\mu}{(\lambda_{_\Sigma}-\mu t)^2},
\end{align}
and
\begin{align}\label{58}
\frac{d^2}{dt^2}(\phi_Y)_{_\Sigma}(t) = \frac{2\mu\lambda_{_\Sigma}t-2\lambda_{_\Sigma}^2}{(\lambda_{_\Sigma}-\mu t)^4}.
\end{align}
Then
\begin{align}\label{59}
m_1 = \frac{d}{dt}\phi_Y(0)(-\mu) = \sum_{\Sigma\subseteq [\![1,n]\!]}e_{_\Sigma}\frac{1}{\lambda_{_\Sigma}},
\end{align}
and
\begin{align}\label{60}
m_2 = \frac{d^2}{dt^2}\phi_Y(0)(-\mu)^2 = \sum_{\Sigma\subseteq [\![1,n]\!]}e_{_\Sigma} \frac{2}{\lambda_{_\Sigma}^2}.
\end{align}
\begin{align}\label{61}
\sigma^2 = m_2 - m_1^2 = \sum_{\Sigma\subseteq [\![1,n]\!]}e_{_\Sigma} \frac{1}{\lambda^2_{_\Sigma}}.
\end{align}
\section*{Conclusion} 
This article introduces and explores the properties of the one-dimensional Clifford Fourier transform (1DCFT), and showcases its practical application in deriving a related inequality. The effectiveness of 1DCFT in probability theory is demonstrated by examining in detail the characteristic function, expected value, and variance within the framework of Clifford algebra. These results represent an important step forward in the development of probability theory using Clifford algebra. The study recommends future research into uncertainty principles concerning the Clifford Algebra probability density function and its characteristic function.

\section*{Declarations}
\textbf{Author Contributions.} This paper is the result of a joint work between the two authors. Both of them read and approved the final manuscript and they are grateful to the editor and the anonymous referees.\\
\\
\textbf{Conflict of interest.} The authors declare that they have no competing interests.\\
\\
\textbf{Availability of data and materials.} Not applicable.\\



\end{document}